\documentclass[12pt]{amsart}
\usepackage{fullpage}
\def\PMS{p}
\def\pmset{{\rm PM}}

\usepackage{amsfonts,amssymb,amscd,amsmath,enumerate,verbatim,color,xcolor}
\usepackage[latin1]{inputenc}
\usepackage{amscd}
\usepackage{latexsym}

\usepackage{graphicx} 

\usepackage{mathptmx}

\usepackage{hyperref}
\hypersetup{colorlinks,linkcolor=blue ,citecolor=blue, urlcolor=blue}
\def\NZQ{\mathbb}               

\def\ZZ{{\NZQ Z}}
\def\RR{{\NZQ R}}

\def\frk{\mathfrak}               

\def\Phi{{\frk N}}
\def\eb{{\mathbf e}}

\def\xb{{\mathbf x}}

\def\opn#1#2{\def#1{\operatorname{#2}}} 
\opn\gr{gr}

\def\Hc{{\mathcal H}}

\def\Pc{{\mathcal P}}
\def\Qc{{\mathcal Q}}

\newtheorem{Theorem}{Theorem}[section]
\newtheorem{Lemma}[Theorem]{Lemma}
\newtheorem{Corollary}[Theorem]{Corollary}
\newtheorem{Proposition}[Theorem]{Proposition}

\theoremstyle{definition}
\newtheorem{Remark}[Theorem]{Remark}

\newtheorem{Example}[Theorem]{Example}

\let\epsilon\varepsilon
\let\phi=\varphi
\let\kappa=\varkappa
\opn\dis{dis}
\opn\height{height}
\opn\dist{dist}
\def\pnt{{\raise0.5mm\hbox{\large\bf.}}}

\opn\Lex{Lex}
\opn\conv{conv}

\begin{document}

\title{PQ-type adjacency polytopes of join graphs}
\author{Hidefumi Ohsugi and Akiyoshi Tsuchiya}
\address{Hidefumi Ohsugi,
	Department of Mathematical Sciences,
	School of Science,
	Kwansei Gakuin University,
	Sanda, Hyogo 669-1337, Japan} 
\email{ohsugi@kwansei.ac.jp}

\address{Akiyoshi Tsuchiya,
Graduate school of Mathematical Sciences,
University of Tokyo,
Komaba, Meguro-ku, Tokyo 153-8914, Japan} 
\email{akiyoshi@ms.u-tokyo.ac.jp}

\subjclass[2010]{05A15, 05C31, 52B12, 52B20}
\keywords{adjacency polytope, $h^*$-polynomial, interior polynomial, perfectly matchable set, join graph, root polytope}

\begin{abstract}
PQ-type adjacency polytopes $\nabla^{\rm PQ}_G$ are lattice polytopes arising from finite graphs $G$. There is a connection between $\nabla^{\rm PQ}_G$  and the engineering problem known as power-flow study, which models the balance of electric power on a network of power generation. 
In particular, the normalized volume of $\nabla^{\rm PQ}_G$ plays a central role.
In the present paper, we focus the case where $G$ is a join graph.
In fact, formulas of the $h^*$-polynomial and the normalized volume of $\nabla^{\rm PQ}_G$ of a join graph $G$ are presented.
Moreover, we give explicit formulas of the $h^*$-polynomial and the normalized volume of $\nabla^{\rm PQ}_G$ when $G$ is a complete multipartite graph or a wheel graph.
\end{abstract}

\maketitle

\section{Introduction}
A \textit{lattice polytope} $\Pc \subset \RR^n$ 
is a convex polytope all of whose vertices have integer coordinates. 
Its \textit{normalized volume}, ${\rm Vol}(\Pc)=\dim (\Pc)! {\rm vol}(\Pc)$ where ${\rm vol}(\Pc)$ is the relative volume of $\Pc$, is always a positive integer. To compute ${\rm Vol}(\Pc)$ is a fundamental but hard problem in polyhedral geometry.

Let $G$ be a simple graph on $[n]:=\{1,\ldots,n\}$ with edge set $E(G)$.
The \textit{PV-type adjacency polytope} $\nabla^{\rm PV}_G$ of $G$ is the lattice polytope which is the convex hull of
\[
\{ \pm (\eb_i-\eb_j) \in \RR^n : \{i,j\} \in E(G)\},
\]
where $\eb_i$ is the $i$-th unit coordinate vector in $\RR^n$.
The normalized volumes of PV-type adjacency polytopes have attracted much attention. 
In fact, the normalized volume of a PV-type adjacency polytope gives an upper bound on the number of possible solutions in the Kuramoto equations (\cite{CDM}), which models the behavior of interacting oscillators (\cite{Kuramoto}).
For several classes of graphs, explicit formulas for the normalized volume of their PV-type adjacency polytopes have been given (e.g., \cite{Ardila,DDM, HJMsymmetric}).
In particular, we can compute the normalized volume of the PV-type adjacency polytope of a suspension graph by using interior polynomials (\cite{OTlocally}).
Here interior polynomials are a version of the Tutte polynomials for hypergraphs introduced by K\'{a}lm\'{a}n \cite{interior}.  
On the other hand,
the \textit{PQ-type adjacency polytope} $\nabla^{\rm PQ}_G$ of $G$ is the lattice polytope which is the convex hull of
\[
\{ (\eb_i,\eb_j) \in \RR^{2n} : \{i,j\} \in E(G) \mbox{ or } i=j\}.
\]
There is a connection between PQ-type adjacency polytopes and the engineering problem known as power-flow study, which models the balances of electric power on a network of power generation (\cite{CM}). 
In fact, the normalized volume of a PQ-type adjacency polytope gives an upper bound on the number of possible solutions in the algebraic power-flow equations. 
Davis and Chen \cite{DC} showed that the normalized volume of $\nabla^{\rm PQ}_G$ can be computed by using sequences of nonnegative integers related to the \textit{Dragon Marriage Problem} \cite{Postnikov}.

In the present paper, we focus on the $h^*$-polynomial of a PQ-type adjacency polytope.
Here, the $h^*$-polynomial $h^*(\Pc,x)$ of a lattice polytope $\Pc$ is a discrete tool to compute the normalized volume ${\rm Vol}(\Pc)$ (see Section \ref{sec:pre}).
We recall a relation between $\nabla^{\rm PQ}_G$ and a root polytope.
For a bipartite graph $H$ on $[n]$ with edge set $E(H)$, the \textit{root polytope} $\Qc_H$ of $H$ is the lattice polytope which is the convex hull of
\[
\{\eb_i+ \eb_j \in \RR^n: \{i,j\} \in E(H)\}.
\]
For a positive integer $n$, set $[\overline{n}]:=\{\overline{1},\ldots,\overline{n}\}$.
Define $D(G)$ to be the bipartite graph on $[n] \cup [\overline{n}]$ with edges $\{i, \overline{i}\}$ for each $i \in [n]$ and $\{i, \overline{j}\}$ and $\{\overline{i},j\}$ for each edge $\{i,j\}$ in $G$. 
It then follows that $\nabla^{\rm PQ}_G$ is unimodularly equivalent to $\Qc_{D(G)}$
(\cite[Lemma 2.4]{DC}).
On the other hand, it is known \cite{KalPos} that the $h^*$-polynomial of the root polytope $\Qc_H$ of a connected bipartite graph $H$ coincides with the interior polynomial $I_H(x)$ of the associated hypergraph of $H$.
In particular, the normalized volume of $\Qc_H$ is equal to  $|{\rm HT}(H)|$, where  ${\rm HT}(H)$ denotes the set of hypertrees of the associated hypergraph of $H$.
Therefore, we can compute the $h^*$-polynomial and the normalized volume of $\nabla^{\rm PQ}_G$ of a connected graph $G$ by using an interior polynomial and counting hypertrees.

The main result of the present paper is formulas of the $h^*$-polynomial and the normalized volume of $\nabla^{\rm PQ}_G$ of a join graph $G$. 
Let $G_1,\ldots,G_s$ be graphs with $m_1,\dots,m_s$ vertices.
Suppose that $G_i$ and $G_j$ have no common vertices for each $i \neq j$.
Then the {\em join} $G_1 + \cdots + G_s$ of $G_1,\ldots,G_s$ is obtained from $G_1 \cup \cdots \cup G_s$
joining each vertex of $G_i$ to each vertex of $G_j$ for any $i \neq j$.
Note that $G_1 + \cdots + G_s$ $(s>1)$ is connected and hence so is
$D(G_1 + \cdots + G_s)$.
For example, the complete bipartite graph $K_{\ell,m}$ is equal to the join $E_\ell + E_m$
where $E_k$ is the empty graph with $k$ vertices.
For complete graphs $K_\ell$ and $K_m$,
one has $K_\ell + K_m = K_{\ell+m}$.
We can compute the $h^*$-polynomial and the normalized volume of the PQ-type adjacency polytope of a join graph by using perfectly matchable set polynomials (see Section~\ref{sec:pre} for the definition of perfectly matchable set polynomials).

\begin{Theorem}
\label{mainthm}
Let $G_1,\ldots,G_s$ be graphs with $m_1,\dots,m_s$ vertices.
Suppose that $G_i$ and $G_j$ have no common vertices for each $i \neq j$.
Then for the join $G= G_1 + \cdots + G_s$ with  $m = \sum_{i=1}^s m_i$ vertices,
we have
\begin{eqnarray}
h^*(\nabla^{\rm PQ}_G ,x)
&=& \sum_{i=1}^s h^*(\nabla^{\rm PQ}_{G_i+K_{m-m_i}},x) -
(s-1)\sum_{k=0}^{m -1} 
\binom{m-1}{k}^2
x^k 
\\
&=& \sum_{i=1}^s I_{D(G_i+K_{m-m_i})}(x) -
(s-1)\sum_{k=0}^{m -1} 
\binom{m-1}{k}^2
x^k \label{main1}\\
&=&\sum_{i=1}^s 
\PMS (D(G_i+K_{m-m_i-1}), x)
-
(s-1)\sum_{k=0}^{m -1} 
\binom{m-1}{k}^2
x^k, \label{main2}
\end{eqnarray}
where $\PMS (H,x)$ denotes the perfectly matchable set polynomial of a graph $H$.
In particular, one has 
\begin{align*}
{\rm Vol}(\nabla^{\rm PQ}_G)&=\sum_{i=1}^s 
{\rm Vol}(\nabla^{\rm PQ}_{G_i+K_{m-m_i}})
-
(s-1)
\binom{2(m-1)}{m-1}\\&=\sum_{i=1}^s 
|{\rm HT}(D(G_i+K_{m-m_i}))|
-
(s-1)
\binom{2(m-1)}{m-1}\\
&=\sum_{i=1}^s 
 |{\rm PM}(D(G_i+K_{m-m_i-1}))|
-
(s-1)
\binom{2(m-1)}{m-1},
\end{align*}
where ${\rm PM}(H)$ denotes the set of perfectly matchable sets of a graph $H$.
\end{Theorem}

By using this theorem, we give explicit formulas of the $h^*$-polynomial and the normalized volume of $\nabla^{\rm PQ}_G$ when $G$ is a complete multipartite graph (Corollary \ref{mainCor}).
On the other hand, Theorem~\ref{mainthm} is not useful for computing the $h^*$-polynomial and the normalized volume of $\nabla^{\rm PQ}_{G}$ when $G$ is a wheel graph $W_n$, that is, $G$ is the join of a cycle $C_n$ and $K_1$.
We give explicit formulas for the $h^*$-polynomial and the normalized volume of $\nabla_{W_n}^{\rm PQ}$
and prove the conjecture \cite[Conjecture~4.4]{DC} on the normalized volume of
$\nabla_{W_n}^{\rm PQ}$ (Theorem \ref{thm:wheel}) by computing the perfectly matchable set polynomial of $D(C_n)$.

The paper is organized as follows:
After reviewing the definitions and properties of the $h^*$-polynomials of lattice polytopes and the interior polynomials of connected bipartite graphs in Section \ref{sec:pre}, we give a proof of Theorem~\ref{mainthm} in Section~\ref{sec:main}.
By using Theorem~\ref{mainthm}, explicit formulas of the $h^*$-polynomial and the normalized volume of the PQ-type adjacency polytope of a complete multipartite graph are presented in Section~\ref{sec:mult}.
Finally, we compute the $h^*$-polynomial and the normalized volume of the PQ-type adjacency polytope of a wheel graph in Section~\ref{sec:wheel}.

\section{preliminaries}
\label{sec:pre}

As explained in the previous section, the $h^*$-polynomial of $\nabla^{\rm PQ}_G$
is equal to the interior polynomial of $D(G)$.
First, we recall what $h^*$-polynomials are.
Let $\Pc \subset \RR^n$ be a lattice polytope of dimension $d$.
Given a positive integer $t$, we define
$$L_{\Pc}(t)=|t \Pc \cap \ZZ^n|,$$
where $t\Pc:=\{ t \xb \in \RR^n : \xb \in \Pc\}$.
The study on $L_{\Pc}(t)$ originated in Ehrhart \cite{Ehrhart} who proved that $L_{\Pc}(t)$ is a polynomial in $t$ of degree $d$ with the constant term $1$.
We call $L_{\Pc}(t)$ the \textit{Ehrhart polynomial} of $\Pc$.
The generating function of the lattice point enumerator, i.e., the formal power series
$$\text{Ehr}_\Pc(x)=1+\sum\limits_{k=1}^{\infty}L_{\Pc}(k)x^k$$
is called the \textit{Ehrhart series} of $\Pc$.
It is known that it can be expressed as a rational function of the form
$$\text{Ehr}_\Pc(x)=\frac{h^*(\Pc,x)}{(1-x)^{d+1}},$$
where $h^*(\Pc,x)$ is a polynomial in $x$ of degree at most $d$ with nonnegative integer coefficients and it
is called
the \textit{$h^*$-polynomial} (or the \textit{$\delta$-polynomial}) of $\Pc$. 
Moreover, $$h^*(\Pc,x)=\sum_{i=0}^{d} h_i^* x^i$$ satisfies $h^*_0=1$, $h^*_1=|\Pc \cap \ZZ^n|-(d +1)$ and $h^*_{d}=|{\rm int} (\Pc) \cap \ZZ^n|$, where ${\rm int} (\Pc)$ is the relative interior of $\Pc$.
Furthermore, $h^*(\Pc,1)=\sum_{i=0}^{d} h_i^*$ is equal to the normalized volume ${\rm Vol}(\Pc)$ of $\Pc$.
We refer the reader to \cite{BeckRobins} for the detailed information about Ehrhart polynomials and $h^*$-polynomials.

Next, we recall the definition of interior polynomials and their properties.
A {\em hypergraph} is a pair $\Hc = (V, E)$, where $E=\{e_1,\ldots,e_n\}$ is a finite multiset of non-empty subsets of $V=\{v_1,\ldots,v_m\}$. 
Elements of $V$ are called vertices and the elements of $E$ are the  hyperedges.
Then we can associate $\Hc$ to a bipartite graph ${\rm Bip} \Hc$
on the vertex set $V \cup E$ with the edge set
$\{ \{v_i, e_j\} : v_i \in e_j\}.$
Assume that ${\rm Bip} \Hc$ is connected.
A {\em hypertree} in $\Hc$ is a function $f: E \rightarrow \ZZ_{\ge 0}$
such that there exists a spanning tree $\Gamma$ of ${\rm Bip} \Hc$ 
whose vertices have degree $f (e) +1$ at each $e \in E$.
Then we say that $\Gamma$ induces $f$.
Let ${\rm HT}(\Hc)$ denote the set of all hypertrees in $\Hc$.
A hyperedge $e_j \in E$ is said to be {\em internally inactive}
with respect to the hypertree $f$ if there exists $j' < j$ such that
$g: E \rightarrow \ZZ_{\ge 0}$ defined by 
\[
 g(e_i) = \begin{cases}
 f(e_i)+1 & (i=j') \\
f(e_i)-1 & (i=j) \\
f(e_i) & \mbox{(otherwise)} 
  \end{cases}
\]
is a hypertree.
Let $\overline{\iota} (f) $ be the number of internally inactive hyperedges of $f$.
Then the {\em interior polynomial} of $\Hc$
is the generating function 
$I_\Hc (x) = \sum_{f \in {{\rm HT}(\Hc)}} x^{ \overline{\iota} (f) }$.
It is known \cite[Proposition~6.1]{interior} that $\deg I_\Hc (x) \le \min\{|V|,|E|\} - 1$.
If $G = {\rm Bip} \Hc$, then we set ${\rm HT}(G) = {\rm HT}(\Hc)$ and $I_G (x) = I_\Hc (x)$.
The coefficients of $I_G(x)$ are described as follows.

\begin{Proposition}[{\cite[Theorem 3.4]{interior}}]
\label{coef_interior}
Let $G$ be a connected bipartite graph on the vertex set $V_1 \sqcup V_2$
where $V_1=\{v_1,\ldots,v_p\}$ and $|V_2|=q$.
Then the coefficient of $x^k$ in $I_G(x)$ is
the number of functions $f: V_1 \rightarrow \ZZ_{\ge 0}$
such that 
\begin{itemize}
\item[{\rm (i)}]
$\displaystyle \sum_{i=1}^p f(v_i) = q -1${\rm ;}

\item[{\rm (ii)}]
$\displaystyle \sum_{v \in V'} f(v) \le |\Gamma_G(V')| -1$ for all $V' \subset V_1$,
where $\Gamma_G(S) \subset V_2$ is the set of vertices adjacent to some vertex in $S${\rm ;}

\item[{\rm (iii)}]
$\overline{\iota} (f) =k$, i.e., $|\eta_G(f)|=k$ where 
$\eta_G(f)$ is the set of vertices $v_j \in V_1$ satisfying the following condition{\rm :}
there exists $ j' < j $ such that
the function $g : V_1 \rightarrow \ZZ_{\ge 0}$ defined by
\[
  g(v_i) = \begin{cases}
    f(v_i)+1 & (i=j') \\
    f(v_i)-1 & (i=j) \\
    f(v_i) & (otherwise) 
  \end{cases}
\]
satisfies condition {\rm (ii)} above.
\end{itemize}
\end{Proposition}

From \cite[Lemma 2.4]{DC} and \cite[Theorems 1.1 and 3.10]{KalPos},
we have the following.

\begin{Proposition}
\label{hpoly=interior}
Let $G$ be a connected graph.
Then $h^*(\nabla_G^{PG},x) = I_{D(G)}(x)$.
In particular, the normalized volume of $\nabla_G^{PG}$
is $I_{D(G)}(1)=|{\rm HT}(D(G))|$.
\end{Proposition}

Let $G$ be a finite graph on the vertex set $V=[n]$.
A {\em $k$-matching} of $G$ is a set of $k$ pairwise non-adjacent edges of $G$.
The {\em matching generating polynomial} of $G$ is
$$g(G,x) =\sum_{k \ge 0} m_G(k) x^k,$$
where $m_G(k)$ is the number of $k$-matchings
of $G$.
A $k$-matching of $G$ is said to be {\em perfect} if $2k =n$.
A subset $S \subset V$ is called a {\em perfectly matchable set} \cite{PMS}
if the induced subgraph of $G$ on the vertex set $S$ has a perfect matching.
Let $\pmset (G, k)$ be the set of all perfectly matchable sets $S$ of $G$ with $|S|=2k$ and $\pmset(G)$ the set of all perfectly matchable sets of $G$.
We regard $\emptyset$ as a perfectly matchable set and we set $\pmset(G,0) = \{\emptyset\}$.
Note that $|\pmset (G, k)| \le m_G(k)$ holds in general.
We call the polynomial 
$$\PMS(G,x) = \sum_{k \geq 0}|\pmset(G,k)| x^k$$ 
the {\em perfectly matchable set polynomial} (PMS polynomial) of $G$. 

\begin{Example}
For the cycle $C_4$ of length $4$, we have
$g(C_4,x) = 2 x^2 + 4x+1$ and
$\PMS (C_4,x) = x^2 + 4x+1$.
On the other hand, if a graph $G$ has no even cycles, then we have
$g(G,x) = \PMS(G,x)$.
\end{Example}

Assume that $G$ is a bipartite graph with a bipartition
$[n] =V_1 \sqcup V_2$.
Then let $\widetilde{G}$ be a connected bipartite graph on $[n+2]$
whose edge set is 
\[
E(\widetilde{G}) = E(G) \cup \{ \{i, n+1\}  : i \in V_1\} \cup \{ \{j, n+2\}  : j \in V_2 \cup \{n+1\}\}.
\]

\begin{Proposition}[{\cite[Proposition 3.4]{OTinterior}}]
\label{prop:matchinginterior}
Let $G$ be a bipartite graph. Then we have
\[
I_{\widetilde{G}}(x)= \PMS(G,x) .
\]
\end{Proposition}

The following Lemma is easy to see.
However it is useful when we apply 
Proposition~\ref{prop:matchinginterior} to such graphs.

\begin{Lemma}
\label{suspensionLemma}
Let $G$ be a graph.
Then we have
$$
D(G+K_1) = \widetilde{D(G)} 
.$$
\end{Lemma}

Moreover, we have

\begin{Proposition}
\label{sus_matchable}
Let $G$ be a graph with $n$ vertices.
Then the $h^*$-polynomial of $\nabla_{G+K_1}^{\rm PQ}$ is 
$
\PMS (D(G), x)
.$
In particular, the normalized volume of $\nabla_{G+K_1}^{\rm PQ}$ is 
$|\pmset(D(G))|$.
\end{Proposition}

\begin{proof}
The assertion follows from Propositions \ref{hpoly=interior}, \ref{prop:matchinginterior}
and Lemma~\ref{suspensionLemma}.
\end{proof}

\begin{Remark}
Given a graph $G$, the number of matchings of $G$ is called {\em Hosoya index} of $G$
and denoted by $Z(G)$.
From Proposition~\ref{sus_matchable}, the normalized volume of 
$\nabla_{G+K_1}^{\rm PQ}$ is at most $Z(D(G))$.
\end{Remark}

Let $G$ be a graph on the vertex set $[n]$. 
Given a subset $S \subset [n]$, we associate the $(0,1)$-vector $\rho(S)=\sum_{i \in S} {\bf e}_i \in \RR^n$.
For example, $\rho(\emptyset) = {\bf 0} \in \RR^n$.
The convex hull of 
$$\{ \rho(S)  : S \in \pmset(G,k) \mbox{ for some } k \}$$
 is called a {\em perfectly matchable subgraph polytope}
(PMS polytope) of $G$.
A system of linear inequalities for a PMS polytope of $G$ was given in \cite{PMS} for bipartite graphs,
and in \cite{PMS2} for arbitrary graphs.

\begin{Proposition}[{\cite[Theorem~1]{PMS}}]
\label{prop:PMS_inequalities}
Let $G$ be a bipartite graph on the vertex set $[n]=V_1 \sqcup V_2$. 
Then the PMS polytope of $G$ is the set of all vectors 
$(x_1,\ldots,x_n) \in \RR^n$ such that
$$
0 \le x_i \le 1 \ \  \mbox{ for each } i=1,2,\ldots,n,
$$
$$
\sum_{i \in V_1} x_i = \sum_{j \in V_2} x_j,
$$
$$
\sum_{i \in S} x_i  \le \sum_{j \in \Gamma_G(S)} x_j \ \  \mbox{ for all } S \subset V_1,
$$
where $\Gamma_G(S) \subset V_2$ is the set of vertices adjacent to some vertex in $S$.
\end{Proposition}

\section{join graphs}
\label{sec:main}

In the present section, we give a proof of Theorem~\ref{mainthm}.
Given a graph $G$ and a nonnegative integer $k$, let
$F_G(k)$ be the set of functions satisfying conditions (i)--(iii) in Proposition~\ref{coef_interior} for $D(G)$.

\begin{Lemma}
Let $G_1$ and $G_2$ be graphs with $m_1$ and $m_2$ vertices, respectively.
Suppose that $G_1$ and $G_2$ have no common vertices.
Then for the join $G= G_1 +G_2$ with $m = m_1+m_2$ vertices,
we have
\begin{eqnarray}
\label{intersectionLemma}
\bigsqcup_{k} F_G(k) = \left( \bigsqcup_{k}  F_{G_1+K_{m_2}}(k) \right) \cap \left(\bigsqcup_{k}  F_{K_{m_1}+G_2}(k)\right),
\end{eqnarray}
\begin{eqnarray}
\label{unionLemma}
\left( \bigsqcup_{k}  F_{G_1+K_{m_2}}(k) \right) \cup \left(\bigsqcup_{k}  F_{K_{m_1}+G_2}(k)\right)  = \bigsqcup_{k} F_{K_m}(k).
\end{eqnarray}
\end{Lemma}

\begin{proof}
Let $V_1 \sqcup V_2$ with $V_1=\{v_1,\ldots, v_{m}\}$
and $V_2=\{v_1',\ldots, v_{m}'\}$ 
denote the common vertex set of 
$D(G)$, $D(G_1+K_{m_2})$, $D(K_{m_1}+G_2)$ and $D(K_m)$ $(=K_{m,m})$.
In addition, let 
$V_{1,1} =\{v_1,\ldots, v_{m_1}\}$, $V_{1,2} =\{v_{m_1+1},\ldots, v_m\}$,
$V_{2,1} =\{v_1',\ldots, v_{m_1}'\}$ and $V_{2,2} =\{v_{m_1+1}',\ldots, v_m'\}$,
where $V_{1,j} \sqcup V_{2,j}$ corresponds to the vertex set of $D(G_j)$ for $j =1,2$.

\bigskip

\noindent
{\bf Case 1} ($D(K_m)$){\bf .} 
Since $ |\Gamma_{D(K_m)}(V')| -1 = m-1$ for all $V' \subset V_1$,
condition (i) implies condition (ii) in Proposition~\ref{coef_interior}.

\bigskip

\noindent
{\bf Case 2} ($D(G_1+K_{m_2})$){\bf .} 
Suppose that $f$ satisfies condition (i).
If $v \in V' \subset V_1$ for some $v \in V_{1,2} $, 
then  condition (ii) holds for $V'$
since $\deg (v) = m$.
Thus, if $f$ satisfies condition (i), then
condition (ii) in Proposition~\ref{coef_interior} is equivalent to the condition
\begin{equation}
\label{dg1}
\sum_{v \in V'} f(v) \le |\Gamma_{D(G_1)}(V')| + m_2 -1 \mbox{ for all } V'  \subset V_{1,1}.
\end{equation}

\noindent
{\bf Case 3} ($D(K_{m_1}+G_2)$){\bf .} 
By the similar argument as in Case 2, if $f$ satisfies condition (i), then
condition (ii) in Proposition~\ref{coef_interior} is equivalent to the condition
\begin{equation}
\label{dg2}
\sum_{v \in V'} f(v) \le |\Gamma_{D(G_2)}(V')| + m_1 -1 \mbox{ for all } V'  \subset V_{1,2}.
\end{equation}

\noindent
{\bf Case 4} ($D(G)$){\bf .} 
Suppose that $f$ satisfies condition (i).
If $v,v' \in V' \subset V_1$ for some $v \in V_{1,1}$ and $v' \in V_{1,2}$,
then condition (ii) holds for $V'$
since $\Gamma_{D(G)}(\{v,v'\}) =V_2$.
Hence, if $f$ satisfies condition (i), then 
condition (ii) in Proposition~\ref{coef_interior} holds if and only if both (\ref{dg1}) and (\ref{dg2}) hold.

\bigskip

\noindent
Hence (\ref{intersectionLemma}) holds.

Since any element in $F_{G_1+K_{m_2}}(k) \cup  F_{K_{m_1}+G_2}(k)$ satisfies condition (i), we have
$$
\left( \bigsqcup_{k}  F_{G_1+K_{m_2}}(k) \right) \cup \left(\bigsqcup_{k}  F_{K_{m_1}+G_2}(k)\right)  \subset \bigsqcup_{k} F_{K_m}(k).
$$
Suppose that $f$ satisfies none of (\ref{dg1}) and (\ref{dg2}).
Then 
$$\sum_{v \in V'} f(v) \ge |\Gamma_{D(G_1)}(V')| + m_2  \mbox{ for some } V'  \subset V_{1,1},$$
$$\sum_{v \in V''} f(v) \ge |\Gamma_{D(G_2)}(V'')| + m_1 \mbox{ for some } V''  \subset V_{1,2}.$$
It then follows that $\sum_{v \in V_1} f(v)  \ge m_1 + m_2 =m$.
Hence, if $f$ satisfies condition (i), then at least one of (\ref{dg1}) or (\ref{dg2}) holds.
Thus (\ref{unionLemma}) holds.
\end{proof}

\begin{Lemma}
\label{bunkaibunkai}
Let $G_1$ and $G_2$ be graphs with $m_1$ and $m_2$ vertices, respectively.
Suppose that $G_1$ and $G_2$ have no common vertices.
Then for the join $G= G_1 +G_2$ with $m = m_1+m_2$ vertices,
$F_{K_m}(k)$ is decomposed into the disjoint sets
$$
F_{K_m}(k)
=
F_{K_{m_1}+G_2}(k) \sqcup (F_{G_1+K_{m_2}}(k) \setminus F_G(k))
$$
as in Fig.~{\rm \ref{zu1}}.
\end{Lemma}
\begin{figure}[h]
\includegraphics[width=10cm,pagebox=cropbox,clip]{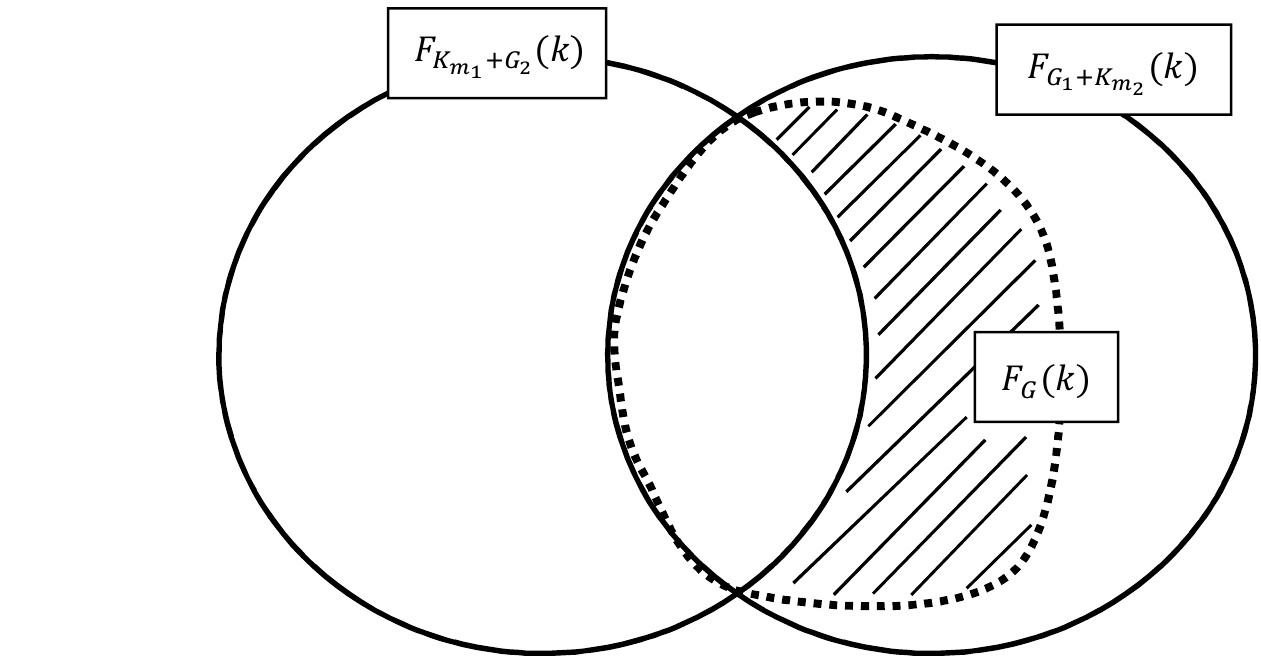}\hspace{2cm} 
\caption{Decomposition of $F_{K_m}(k)$}
\label{zu1}
\end{figure}

\begin{proof}
Let $V_1 \sqcup V_2$ with $V_1=\{v_1,\ldots, v_{m}\}$
and $V_2=\{v_1',\ldots, v_{m}'\}$ 
denote the common vertex set of 
$D(G)$, $D(G_1+K_{m_2})$, $D(K_{m_1}+G_2)$ and $D(K_m)$ $(=K_{m,m})$.
In addition, let 
$V_{1,1} =\{v_1,\ldots, v_{m_1}\}$, $V_{1,2} =\{v_{m_1+1},\ldots, v_m\}$,
$V_{2,1} =\{v_1',\ldots, v_{m_1}'\}$ and $V_{2,2} =\{v_{m_1+1}',\ldots, v_m'\}$,
where $V_{1,j} \sqcup V_{2,j}$ corresponds to the vertex set of $D(G_j)$ for $j =1,2$.

\bigskip

\noindent
{\bf Claim 1.} 
$F_{K_{m_1}+G_2}(k) \subset F_{K_m}(k)$.

Suppose that $f$ belongs to $F_{K_{m_1}+G_2}(k) \cup F_{K_m}(k)$.
Then
condition (ii) in Proposition~\ref{coef_interior} is independent from
the value $f(v_1)$ since $\deg (v_1) =m$.
Thus one can choose $j'=1$ for condition (iii) for any $j >1$, and hence we have
$$k = |\{v_j \in V_1 : v_j \ne v_1, f(v_j) >0\}|.$$
Therefore
\begin{equation}
F_{K_{m_1}+G_2}(k) \subset F_{K_m}(k) \label{hougan01}
\end{equation}
 and $F_{K_{m_1}+G_2}(\ell) \cap F_{K_m}(k) =\emptyset$
if $\ell \ne k$.

\bigskip

\noindent
{\bf Claim 2.} 
$F_G(k) \subset  F_{G_1+K_{m_2}}(k)$.

Suppose that $f \in F_G(k)$.
Then $f \in F_{G_1+K_{m_2}}(\ell)$ for some $\ell$.
Note that $k = |\eta_{D(G)} (f)|$ and $\ell = |\eta_{D(G_1+K_{m_2})} (f)|$
(where $\eta$ is defined in  Proposition~\ref{coef_interior}).
In order to prove $k=\ell$, we will show that
$\eta_{D(G)} (f) = \eta_{D(G_1+K_{m_2})} (f)$.
Suppose that $f(v_j) >0$ for $v_j \in V_1$ ($j \ne 1$).
If $j \le m_1+1$, then $v_j$ 
belongs to $\eta_{D(G)}(f)$ 
if and only if $v_j$ belongs to 
$\eta_{D(G_1+K_{m_2})} (f)$ since the vertex $v_{j'}$ in condition (iii) in Proposition~\ref{coef_interior}
should be chosen from $V_{1,1}$.
Let $j \ge m_1+2$.
Since 
$v_{m_1+1}$ belongs to $V_{1,2}$ with $m_1+1 < j$, 
(\ref{dg1}) holds for $g$ defined in 
condition (iii) of Proposition~\ref{coef_interior} where $v_{j'}=v_{m_1+1}$.
Thus $v_{j'}=v_{m_1+1}$ satisfies condition (iii) of Proposition~\ref{coef_interior}
for $v_j$ in $D(G_1+K_{m_2})$.
Hence $v_j$ belongs to 
$\eta_{D(G_1+K_{m_2})} (f)$.
We now show that
any $v_j$ ($j \ge m_1+2$) belongs to 
$\eta_{D(G)} (f)$, i.e.,
there exists $j' < j$ such that
$g: V_1 \rightarrow \ZZ_{\ge 0}$ defined by 
\begin{equation}
 g(v_i) = \begin{cases}
 f(v_i)+1 & (i=j') \\
f(v_i)-1 & (i=j) \\
f(v_i) & \mbox{(otherwise)} 
  \end{cases} \label{findg}
\end{equation}
is a hypertree in $D(G)$.
Let $\Gamma$ be a spanning tree of $D(G)$ that induces $f$.
Suppose that $\Gamma$ does not contain an edge $e=\{v_k, v_1'\}$ for some $m_1< k \le m$.
Then $\Gamma \cup \{e\}$ has a unique cycle, and the cycle contains $v_k$.
Let $e'$ be the edge of the cycle that is adjacent to $v_k$
 but different from $e$.
Then $(\Gamma \cup \{e'\}) \setminus \{e\}$ induces $f$.
Thus we may assume that $\{ \{v_k, v_1'\} :   m_1 <k\le m \} \subset \Gamma$.
By the same argument, 
we may assume that 
$\{ \{v_k, v_{m_1+1}'\} :  1\le k \le m_1\}$ is a subset of $\Gamma$.
Hence
$$S=\{ \{v_1, v_{m_1+1}'\} ,\ldots, \{v_{m_1}, v_{m_1+1}'\} ,
\{v_{m_1+1},v_1'\}, \ldots, \{v_{m},v_1'\} \}$$
is a subset of $\Gamma$.
Since $\Gamma$ is spanning, 
$$T=
\{
\{v_{i_2}, v_2'\} , \ldots, \{v_{i_{m_1}}, v_{m_1}'\},
\{v_{i_{m_1+2}}, v_{m_1+2}'\} , \ldots, \{v_{i_{m}}, v_m'\} 
\}$$
is a subset of $\Gamma$ for some $1 \le i_2, \ldots, i_{m_1}, i_{m_1+2}, \ldots,i_m \le m$.
Then $S\cap T = \emptyset$ and $|S \cup T| = 2m-2$.
Since $\Gamma$ has $2m-1$ edges, we have
$$
\Gamma= S \sqcup T \sqcup \{e\}
$$
for some edge $e$ of $D(G)$.
In particular, the degree of each vertex $v_i'$ ($i \in [m] \setminus \{1,m_1+1\}$) in $\Gamma$ is at most $2$.
Since $f(v_j) >0$, the degree of $v_j$ in $\Gamma$ is at least 2,
 and hence $\{v_j,v_k'\}$ is an edge of $\Gamma$ for some $k \in [m]\setminus \{1\}$.
Let $e'=\{v_j,v_k'\}$.
We now construct a spanning tree of $D(G)$ that induces a hypertree of the form (\ref{findg}).
Note that, for an edge $e''$ of $D(G)$, 
$\Gamma'=(\Gamma \cup \{e''\} ) \setminus \{e'\}$ is a spanning tree of $D(G)$ 
if $e'' \notin \Gamma$ and the (unique) cycle of $\Gamma \cup \{e''\} $ contains $e'$.

\bigskip

\noindent
{\bf Case A} ($k =m_1+1$){\bf .}
Then $\Gamma$ contains a path $(v_1', v_j, v_{m_1+1}', v_1)$.
Since $\Gamma$ has no cycles, $\{v_1,v_1'\}$ is not an edge of $\Gamma$.
Hence $(\Gamma \cup \{v_1, v_1'\} ) \setminus \{e'\}$
is a spanning tree of $D(G)$ that induces a hypertree of the form (\ref{findg}) where $j'=1$.

\bigskip

\noindent
{\bf Case B} ($k \ne m_1+1$ and $\deg (v_k') =1$){\bf .}
If $2 \le k \le m_1$, then $(\Gamma \cup \{v_k, v_k'\} ) \setminus \{e'\}$
is a spanning tree of $D(G)$ that induces a hypertree of the form (\ref{findg}) where $j'=k$.
If $m_1+2 \le k \le m$, then $(\Gamma \cup \{v_1, v_k'\} ) \setminus \{e'\}$
is a spanning tree of $D(G)$ that induces a hypertree of the form (\ref{findg}) where $j'=1$.

\bigskip

\noindent
{\bf Case C} ($k \ne m_1+1$ and $\deg (v_k') =2$){\bf .}
Then $\Gamma$ contains an edge $\{v_\ell,v_k'\}$ for some $\ell \ne j$.
If $m_1 < \ell \le m$, then $\Gamma$ contains a cycle $(v_1', v_j, v_k', v_\ell, v_1')$ of length 4.
This is a contradiction.
Hence $\ell \le m_1$.
Then $\Gamma$ contains a path $(v_1', v_j, v_k', v_\ell ,v_{m_1+1}', v_1)$.
Since $\Gamma$ has no cycles, $\{v_1,v_1'\}$ is not an edge of $\Gamma$.
Hence $(\Gamma \cup \{v_1, v_1'\} ) \setminus \{e'\}$
is a spanning tree of $D(G)$ that induces a hypertree of the form (\ref{findg}) where $j'=1$.

\bigskip

\noindent
Thus we have $k=\ell$, and hence
\begin{equation}
\label{bubun1}
F_G(k) \subset  F_{G_1+K_{m_2}}(k).
\end{equation}

\bigskip

\noindent
{\bf Claim 3.} 
$(F_{G_1+K_{m_2}}(k) \setminus F_G(k)) \subset F_{K_m}(k)$.

Suppose that $f$ belongs to $F_{G_1+K_{m_2}}(k) \setminus F_G(k)$.
Then there exists $V' \subset V_{1,2}$ such that
$$
\sum_{v \in V'} f(v) \ge |\Gamma_{D(G_2)}(V')| +m_1.
$$
In particular, $\sum_{v \in V_{1,2}} f(v) > m_1$.
Since $\sum_{v \in V_1} f(v) = m -1$, we have $\sum_{v \in V_{1,1}} f(v) < m_2 - 1$.
Hence
$$
\sum_{v \in V''} f(v) < m_2 - 1 <  |\Gamma_{D(G_1)}(V'')| + m_2 - 1
$$
for all $V''  \subset V_{1,1}$.
Thus, for each $v_j \ne v_1$ with $f(v_j) >0$,
$g: V_1 \rightarrow \ZZ_{\ge 0}$ defined by 
\begin{equation*}
 g(v_i) = \begin{cases}
 f(v_i)+1 & (i=1) \\
f(v_i)-1 & (i=j) \\
f(v_i) & \mbox{(otherwise)} 
  \end{cases} 
\end{equation*}
satisfies
$$
\sum_{v \in V''} g(v) \le
1+\sum_{v \in V''} f(v) \le  |\Gamma_{D(G_1)}(V'')| + m_2 - 1
$$
for all $V''  \subset V_{1,1}$.
Therefore we have
$$k = |\{v_j \in V_1 : v_j \ne v_1, f(v_j) >0\}|,$$
and hence
\begin{equation}
(F_{G_1+K_{m_2}}(k) \setminus F_G(k)) \subset F_{K_m}(k). \label{hougan02}
\end{equation}

\bigskip

\noindent
{\bf Claim 4.} 
$F_{K_m}(k)
=
F_{K_{m_1}+G_2}(k) \cup (F_{G_1+K_{m_2}}(k) \setminus F_G(k))$.

From (\ref{hougan01}) and (\ref{hougan02}),
$$
F_{K_m}(k)
\supset
F_{K_{m_1}+G_2}(k) \cup (F_{G_1+K_{m_2}}(k) \setminus F_G(k)).
$$
Let $f \in F_{K_m}(k) \setminus F_{K_{m_1}+G_2}(k) $.
Since $F_{K_{m_1}+G_2}(k') \cap F_{K_m}(k) =\emptyset$
for any $k' \ne k$, it follows that
$f \notin F_{K_{m_1}+G_2}(\ell) $ for any $\ell$.
Hence from (\ref{unionLemma}),
 $f \in F_{G_1+K_{m_2}}(\ell)$ for some $\ell$.
If $f \in F_G(\ell)$, then 
$f \in F_{K_{m_1}+G_2}(\ell') $ for some $\ell'$
by (\ref{intersectionLemma}).
This is a contradiction.
Thus $f \in F_{G_1+K_{m_2}}(\ell) \setminus  F_G(\ell)$.
From (\ref{hougan02}), we have
$$
f \in  F_{G_1+K_{m_2}}(\ell) \setminus  F_G(\ell)
\subset 
F_{K_m}(\ell). 
$$
Then $\ell =k$.
It follows that
$$
F_{K_m}(k)
=
F_{K_{m_1}+G_2}(k) \cup (F_{G_1+K_{m_2}}(k) \setminus F_G(k)).
$$

Finally, we show that this is a decomposition.
Suppose that $f \in  F_{G_1+K_{m_2}}(k)   \cap  F_{K_{m_1}+G_2}(k)$.
From (\ref{intersectionLemma}), $f \in F_G(\ell)$ for some $\ell$.
Moreover, from (\ref{bubun1}), we have $\ell = k$.
Thus
$$
 F_{G_1+K_{m_2}}(k)  \cap  F_{K_{m_1}+G_2}(k) \subset  F_G(k).
$$
Therefore $F_{K_m}(k)$ is decomposed into the disjoint sets
$$
F_{K_m}(k)
=
F_{K_{m_1}+G_2}(k) \sqcup (F_{G_1+K_{m_2}}(k) \setminus F_G(k))
$$
 as in Fig.~\ref{zu1}.
\end{proof}

Now, we are in the position to give a proof of Theorem~\ref{mainthm}.
\begin{proof}[Proof of Theorem~\ref{mainthm}]
We prove this by induction on $s$.
First we discuss the case when $s=2$, i.e., $G= G_1+G_2$.
It is known \cite[Example~5.3]{KalPos} that 
$$|F_{K_m}(k)| = \binom{m - 1}{k}^2.$$
From Lemma~\ref{bunkaibunkai},
$$
\binom{m - 1}{k}^2 = |F_{K_{m_1}+G_2}(k)| + (|F_{G_1+K_{m_2}}(k) | - |F_G(k)| ).
$$
Thus we have
$$
I_{D(G)}(x)=I_{D(G_1+K_{m_2})}(x) + I_{D(G_2+K_{m_1})}(x) -\sum_{k=0}^{m -1} 
\binom{m-1}{k}^2
x^k.
$$

Let $s>2$ and assume that the assertion holds for the join of at most $s-1$ graphs.
Since $G = (G_1 + \cdots + G_{s-1}) + G_s$, we have
\begin{eqnarray*}
I_{D(G)}(x) &=&
I_{D((G_1 + \cdots + G_{s-1})+K_{m_s})}(x) + I_{D(G_s+K_{m - m_s})}(x) -\sum_{k=0}^{m -1} 
\binom{m-1}{k}^2
x^k\\
&=&
\sum_{i=1}^{s-2} I_{D(G_i+K_{m-m_i})}(x) 
+
I_{D((G_{s-1}+K_{m_s})+K_{m-m_{s-1}-m_s})}(x) 
\\
& & -
(s-2)\sum_{k=0}^{m -1} 
\binom{m-1}{k}^2
x^k+ I_{D(G_s+K_{m - m_s})}(x) -\sum_{k=0}^{m -1} 
\binom{m-1}{k}^2
x^k\\
&=& \sum_{i=1}^s I_{D(G_i+K_{m-m_i})}(x) -
(s-1)\sum_{k=0}^{m -1} 
\binom{m-1}{k}^2
x^k.
\end{eqnarray*}
From Proposition~\ref{prop:matchinginterior}
and Lemma~\ref{suspensionLemma},
this is equal to equation (\ref{main2}). 
\end{proof}

\section{Complete multipartite graphs}
\label{sec:mult}

In this section, 
applying Theorem~\ref{mainthm}, 
we give explicit formulas for the $h^*$-polynomial 
and the normalized volume of
the PQ-type adjacency polytope of a complete multipartite graph.

Given positive integers $\ell$ and $m$,
let
$$
f_{\ell,m}(x) = \sum_{k=0}^{\ell + m-1} 
\sum_{\alpha =0}^{k} 
\binom{\ell-1}{k-\alpha}
\binom{m}{\alpha}
\sum_{\beta=\alpha}^{k}
\binom{\ell+\alpha-1}{\beta}
\binom{m-\alpha}{k-\beta}
x^k.
$$
Since
\begin{eqnarray*}
&&
\sum_{k=0}^{\ell + m-1} 
\sum_{\alpha =0}^{k} 
\binom{\ell-1}{k-\alpha}
\binom{m}{\alpha}
\sum_{\beta=0}^{k}
\binom{\ell+\alpha-1}{\beta}
\binom{m-\alpha}{k-\beta}x^k\\
&=&
\sum_{k=0}^{\ell + m-1} 
\sum_{\alpha =0}^{k} 
\binom{\ell-1}{k-\alpha}
\binom{m}{\alpha}
\binom{\ell+m-1}{k}x^k\\
&=&
\sum_{k=0}^{\ell + m-1} 
\binom{\ell+m-1}{k}
\sum_{\alpha =0}^{k} 
\binom{\ell-1}{k-\alpha}
\binom{m}{\alpha}x^k
\\
&=&
\sum_{k=0}^{\ell + m-1} 
\binom{\ell+m-1}{k}^2x^k
\end{eqnarray*}
holds,  we have
$$
f_{\ell,m}(x) = 
\sum_{k=0}^{\ell+m -1} 
\binom{\ell+m-1}{k}^2 x^k
-\sum_{k=1}^{m-1} 
\sum_{\alpha =1}^{k} 
\binom{\ell-1}{k-\alpha}
\binom{m}{\alpha}
\sum_{\beta=0}^{\alpha-1}
\binom{\ell+\alpha-1}{\beta}
\binom{m-\alpha}{k-\beta}
x^k.
$$
The $h^*$-polynomial of $\nabla_G^{\rm PQ}$ for the graph
$G= K_\ell +E_m $ $(=K_{1,\dots,1,m})$ coincides with $f_{\ell,m} (x)$.

\begin{Theorem}
\label{summand}
Let $G= K_\ell +E_m $.
Then we have
\begin{eqnarray*}
h^*(\nabla_G^{\rm PQ},x) &=&f_{\ell,m}(x),\\
{\rm Vol}(\nabla_G^{\rm PQ}) &=& f_{\ell,m}(1)
=
\sum_{\alpha =0}^{m} 
\binom{m}{\alpha}
\sum_{\beta=0}^{\ell-1}
\binom{\ell+\alpha-1}{\beta}
\binom{\ell+m-\alpha-1}{\beta}.
\end{eqnarray*}
\end{Theorem}

\begin{proof}
Let $G'=D(K_{\ell-1}+E_m)$.
Since $G =(K_{\ell-1}+E_m) +K_1$,
from Proposition~\ref{sus_matchable}, we have
$$
h^*(\nabla_G^{\rm PQ},x)=
\PMS(G',x)=
\sum_{k = 0}^{\ell + m-1} \left|\pmset(G' ,k)\right| x^k.
$$
Let $n =\ell+m-1$
and let $[n] \cup [\overline{n}]$ be the vertex set of $G'$.
We decompose $[n]$ into two disjoint set $V_{1,1} = [\ell-1]$ and $V_{1,2} = [n]\setminus [\ell-1]$ where $V_{1,1}$ (resp.~$V_{1,2}$) corresponds to $K_{\ell-1}$
(resp.~$E_m$).
From Proposition~\ref{prop:PMS_inequalities},
each $|\pmset(G' ,k)|$ is the number of $(0,1)$-vectors
$(x_1,\ldots,x_n, y_1,\ldots,y_n) \in \RR^{2n}$ such that
\begin{eqnarray}
\sum_{i = 1}^n x_i &=& \sum_{i = 1}^n y_i = k, \label{wa0}\\ 
\sum_{i \in S} x_i  &\le& \sum_{\overline{j} \in \Gamma_{G'}(S)} y_j \ \  \mbox{ for all } S \subset [n].\label{ue0}
\end{eqnarray}
If (\ref{wa0}) holds and a subset $S \subset [n]$ contains an element of $V_{1,1}$, then we have $\Gamma_{G'}(S)=[\overline{n}]$, and 
hence 
$$ \sum_{\overline{j} \in \Gamma_{G'}(S)} y_j=k
=\sum_{i = 1}^n x_i \ge \sum_{i \in S} x_i 
.$$
Thus (\ref{wa0}) and (\ref{ue0}) hold if and only if
\begin{eqnarray}
\sum_{i = 1}^n x_i &=& \sum_{i = 1}^n y_i = k, \label{wa}\\ 
\sum_{i \in S} x_i  &\le& \sum_{\overline{j} \in \Gamma_{G'}(S)} y_j \ \  \mbox{ for all } S \subset V_{1,2}. \label{ue}
\end{eqnarray}
We count such vectors with $|\{ i \in V_{1,2} : x_i = 1 \}|=\alpha$
for each $\alpha = 0,1,\ldots, k$.
There are $\binom{m}{\alpha}$ possibilities for the choice of the subset 
$S'=\{ i \in V_{1,2} : x_i = 1 \} \subset V_{1,2}$.
For each  $S'$, there are $\binom{\ell-1}{k-\alpha}$ 
possibilities for the choice of the subset 
$\{ i \in V_{1,1} : x_i = 1 \} \subset V_{1,1}$.
Then equations (\ref{wa}) and (\ref{ue}) hold if and only if 
$\sum_{i = 1}^n y_i = k$ and $\alpha \le \sum_{\overline{j} \in \Gamma_{G'}(S')} y_j$.
Let $\beta = \sum_{\overline{j} \in \Gamma_{G'}(S')} y_j$.
Since $|\Gamma_{G'}(S')| = \ell + \alpha -1$, there are 
$\binom{\ell+\alpha-1}{\beta}$ possibilities for the choice of the subset 
$S''=\{ \overline{j} \in \Gamma_{G'}(S') : y_j = 1 \} \subset\Gamma_{G'}(S')$.
For each  $S''$, there are $\binom{m-\alpha}{k-\beta}$ 
possibilities for the choice of the subset 
$\{ \overline{j} \in [\overline{n}] \setminus \Gamma_{G'}(S') : y_j = 1 \} \subset[\overline{n}] \setminus \Gamma_{G'}(S')$.
Thus we have $I_{D(G)}(x)=f_{\ell,m}(x)$.

Moreover,  the normalized volume of $\nabla_G^{\rm PQ}$ is equal to
\begin{eqnarray*}
f_{\ell,m}(1)
&=&
\sum_{k=0}^{\ell + m-1} 
\sum_{\alpha =0}^{k} 
\binom{\ell-1}{k-\alpha}
\binom{m}{\alpha}
\sum_{\beta=\alpha}^{k}
\binom{\ell+\alpha-1}{\beta}
\binom{m-\alpha}{k-\beta}\\
&=&
\sum_{\alpha =0}^{m} 
\binom{m}{\alpha}
\sum_{\beta=\alpha}^{\ell+\alpha-1}
\binom{\ell+\alpha-1}{\beta}
\sum_{k=\beta}^{\ell + \alpha-1} 
\binom{\ell-1}{k-\alpha}
\binom{m-\alpha}{k-\beta}\\
&=&
\sum_{\alpha =0}^{m} 
\binom{m}{\alpha}
\sum_{\beta=\alpha}^{\ell+\alpha-1}
\binom{\ell+\alpha-1}{\beta}
\sum_{k=\beta}^{\ell + \alpha-1} 
\binom{\ell-1}{\ell-1+\alpha-\beta-(k-\beta)}
\binom{m-\alpha}{k-\beta}\\
&=&
\sum_{\alpha =0}^{m} 
\binom{m}{\alpha}
\sum_{\beta=\alpha}^{\ell+\alpha-1}
\binom{\ell+\alpha-1}{\beta}
\binom{\ell+m-\alpha-1}{\ell+\alpha-\beta-1}\\
&=&
\sum_{\alpha =0}^{m} 
\binom{m}{\alpha}
\sum_{\beta'=0}^{\ell-1}
\binom{\ell+\alpha-1}{\beta'}
\binom{\ell+m-\alpha-1}{\beta'}.
\end{eqnarray*}
\end{proof}

\begin{Remark}
Let $G= K_\ell +E_m $.
Since
$$
\sum_{\alpha =0}^{m} 
\binom{m}{\alpha}
\sum_{\beta=0}^{\ell+\alpha-1}
\binom{\ell+\alpha-1}{\beta}
\binom{\ell+m-\alpha-1}{\beta}
=
\binom{2(\ell+m-1)}{\ell+m-1},
$$
we have
$$
{\rm Vol}(\nabla_G^{\rm PQ}) =
\binom{2(\ell+m-1)}{\ell+m-1}-
\sum_{\alpha =1}^{m-1} 
\binom{m}{\alpha}
\sum_{\beta=\ell}^{\ell+\alpha-1}
\binom{\ell+\alpha-1}{\beta}
\binom{\ell+m-\alpha-1}{\beta}.
$$
\end{Remark}

\begin{Example}
\label{exaofexa}
For small $\ell$ and $m$, $f_{\ell, m}(1)$ in Theorem~\ref{summand} is
\begin{eqnarray*}
f_{1,m}(1) &=& 2^m,\\
f_{2,m}(1) &=& 2^{m-2} \left(m^2+3 m+8\right),\\
f_{3,m}(1) &=& \frac{2^{m-4} }{(2!)^2} \left(m^4+10 m^3+59 m^2+186 m+384\right),\\
f_{4,m}(1) &=& \frac{2^{m-6}}{(3!)^2}  \left(m^6+21 m^5+229 m^4+1563 m^3+7762 m^2+24984 m+46080\right),\\
\\
f_{\ell,1}(1) &=& \binom{2\ell}{\ell},\\
f_{\ell,2}(1) &=& \binom{2(\ell+1)}{\ell+1}-2,\\
f_{\ell,3}(1) &=& \binom{2(\ell+2)}{\ell+2}-(6\ell+6),\\
f_{\ell,4}(1) &=& \binom{2(\ell+3)}{\ell+3}-(10\ell^2+24\ell +20).
\end{eqnarray*}
\end{Example}

From Theorems~\ref{mainthm} and \ref{summand}, we have the following.

\begin{Corollary}
\label{mainCor}
Let $G$ be a complete multipartite graph $K_{m_1,\dots,m_s}$,
and let $m= \sum_{i=1}^s m_i$.
Then we have
\begin{eqnarray*}
h^*(\nabla_G^{\rm PQ},x)&=& 
\sum_{i=1}^s f_{m-m_i,m_i}(x)
- (s-1)\sum_{k=0}^{m -1} 
\binom{m-1}{k}^2
x^k,\\
{\rm Vol} (\nabla_G^{\rm PQ}) &=&
\sum_{i=1}^s f_{m-m_i,m_i}(1) - (s-1)
\binom{2(m-1)}{m-1}.
\end{eqnarray*}
\end{Corollary}

\begin{Example}
From Corollary~\ref{mainCor}, the normalized volume of 
$\nabla_{K_{\ell,m}}^{\rm PQ}$ is
$$
f_{\ell,m}(1)+f_{m,\ell}(1) - 
\binom{2(\ell+m-1)}{\ell + m-1}.
$$
From Example~\ref{exaofexa}, 
we have
\begin{eqnarray*}
{\rm Vol} (\nabla_{K_{1,m}}^{\rm PQ})&=& 2^m,\\
{\rm Vol} (\nabla_{K_{2,m}}^{\rm PQ})&=& 2^{m-2} \left(m^2+3 m+8\right) -2,\\
{\rm Vol} (\nabla_{K_{3,m}}^{\rm PQ})&=& \frac{2^{m-4} }{(2!)^2} \left(m^4+10 m^3+59 m^2+186 m+384\right)-(6m+6),\\
{\rm Vol} (\nabla_{K_{4,m}}^{\rm PQ}) &=& \frac{2^{m-6}}{(3!)^2}  \left(m^6+21 m^5+229 m^4+1563 m^3+7762 m^2+24984 m+46080\right)\\
&&-(10m^2+24m +20).
\end{eqnarray*}
The formula for ${\rm Vol} (\nabla_{K_{2,m}}^{\rm PQ})$ coincides with 
that in \cite[Proposition 4.2]{DC}.
\end{Example}

\begin{Example}
Let $G$ be the complete bipartite graph $K_{2,n-2}$.
Since 
\begin{eqnarray*}
f_{2,m}(x)
&=&
\sum_{k=0}^{m+1} 
\sum_{\alpha =0}^{k} 
\binom{1}{k-\alpha}
\binom{m}{\alpha}
\sum_{\beta=\alpha}^{k}
\binom{\alpha+1}{\beta}
\binom{m-\alpha}{k-\beta}x^k\\
&=&
\sum_{k=0}^{m+1} 
\binom{m}{k-1}
k
(m-k+1)x^k
+
\sum_{k=0}^{m+1} 
\binom{m}{k-1}
x^k
+
\sum_{k=0}^{m+1} 
\binom{m}{k}
(k+1)
x^k\\
&=&
\sum_{k=0}^{m} 
\binom{m}{k}
k^2
x^k
+
\sum_{k=0}^{m} 
\binom{m}{k}
x^{k+1}
+
\sum_{k=0}^{m} 
\binom{m}{k}
(k+1)
x^k\\
&=&
\sum_{k=0}^{m} 
\binom{m}{k}
(x+1+k^2+k)
x^k\\
&=&
(x+1)^{m+1} + m  \left( (m + 1)x +2 \right) x(x+1)^{m-2} 
\end{eqnarray*}
and
\begin{eqnarray*}
f_{\ell,2}(x)
&=&
\sum_{k=0}^{\ell+1} 
\binom{\ell+1}{k}^2 x^k
-2x,
\end{eqnarray*}
we have
\begin{eqnarray*}
h^*(\nabla_G^{\rm PQ},x)  &=& f_{2,n-2}(x) + f_{n-2,2}(x) - \sum_{k=0}^{n -1} 
\binom{n-1}{k}^2
x^k\\
&=& (x+1)^{n-1} +  (n-2) \left( (n- 1)x +2 \right)x (x+1)^{n-4} -2x.
\end{eqnarray*}
\end{Example}

\section{Wheel graphs}
\label{sec:wheel}

For $n\ge 3$, the {\em wheel graph} $W_n$ with $n+1$ vertices
is the join graph $W_n = C_n + K_1$.
Unfortunately, Theorem~\ref{mainthm} is not useful for computing
the $h^*$-polynomial of $\nabla_{W_n}^{\rm PQ}$.
We will give an explicit formula for the $h^*$-polynomial of $\nabla_{W_n}^{\rm PQ}$
and prove the conjecture \cite[Conjecture~4.4]{DC} on the normalized volume of
$\nabla_{W_n}^{\rm PQ}$ by using Proposition~\ref{sus_matchable} on $\nabla_{G+K_1}^{\rm PQ}$.

Let
$$
\gamma(n,x)=
\left\{
\begin{array}{ll}
1 & \mbox{if } n=0,\\
\\
 \frac{
(1+\sqrt{1+8x})^n + (1-\sqrt{1+8x})^n
}{2^n} & \mbox{if } n \mbox{ is odd,}\\
\\
 \frac{
(1+\sqrt{1+8x})^n + (1-\sqrt{1+8x})^n
}{2^n}  -2 x^{\frac{n}{2} } & \mbox{otherwise}.
\end{array}
\right.
$$
For $n\ge 3$, 
$$
\frac{
(1+\sqrt{1+8x})^n + (1-\sqrt{1+8x})^n
}{2^n} = g(C_n, 2x),
$$
where $g(C_n,x)$ is the matching generating polynomial of $C_n$.
Moreover, for $n\ge 3$, 
it is known \cite[Example~4.5]{saboten} that, 
$\gamma(n,x)$ is the $\gamma$-polynomial of 
the PV-type adjacency polytope $\nabla^{\rm PV}_{W_n}$ of $W_n$.
(Note that $\nabla^{\rm PV}_{W_n}$ is called {\em the symmetric edge polytope of type A} of $W_n$ in \cite{saboten}.)
The $h^*$-polynomial of $\nabla_{W_n}^{\rm PQ}$ is described by this function as follows.

\begin{Theorem}
\label{thm:wheel}
The $h^*$-polynomial of $\nabla_{W_n}^{\rm PQ}$ is 
\begin{eqnarray*}
 & & \sum_{k=0}^n \binom{n}{k} \gamma(n-k,x) x^k\\
&=&
\left(\frac{1+2x+\sqrt{1+8x}}{2}\right)^n +
\left(\frac{1+2x-\sqrt{1+8x}}{2}\right)^n 
+x^n
-
(x+\sqrt{x})^n 
- (x-\sqrt{x})^n.
\end{eqnarray*}
Moreover, the normalized volume of $\nabla_{W_n}^{\rm PQ}$ is $3^n-2^n+1$.
\end{Theorem}

\begin{proof}
Since $W_n = C_n + K_1$, the $h^*$-polynomial of $\nabla_{W_n}^{\rm PQ}$ is 
$$
\PMS (D(C_n),x) =
\sum_{\ell = 0}^n|\pmset(D(C_n),\ell)| x^\ell$$ by Proposition~\ref{sus_matchable}.
From Proposition~\ref{prop:PMS_inequalities}, $|\pmset(D(C_n),\ell)|$
is equal to
the number of $(0,1)$-vectors
$(x_1,\ldots,x_n, y_1,\ldots,y_n) \in \RR^{2n}$ such that
\begin{eqnarray}
\sum_{i =1}^{n} x_i &=& \sum_{j=1}^{n} y_j = \ell, \label{wa2}\\ 
\sum_{i \in S} x_i  &\le& \sum_{\overline{j} \in \Gamma_{G'}(S)} y_j \ \  \mbox{ for all } S \subset [n], \label{ue2}
\end{eqnarray}
where $V= [n] \cup [\overline{n}]$ is the set of vertices of $D(C_n)$.
Let $C_n=(1,2,\dots,n,1)$.
Given a subset $T \subset [n]$ and an integer $\ell\in [n]$,
let $\pmset_{T,\ell}$ denote the set of all $(0,1)$ vectors $(x_1,\ldots,x_n, y_1,\ldots,y_n) \in \RR^{2n}$
satisfying (\ref{wa2}), (\ref{ue2}), and $T = \{ i \in [n] : x_i = y_i =1 \}$.
Note that $\pmset_{T,\ell} = \emptyset$ if $\ell < |T|$.

Let $T \subset [n]$ with $|T| =k$.
We will show that $\sum_{\ell=k}^n |\pmset_{T,\ell}| x^\ell = \gamma(n-k,x) x^k$.

\bigskip

\noindent
{\bf Case 1} ($k=n$).
It is easy to see that 
$\pmset_{T,n}= \{(1,\dots,1)\}$ and $\pmset_{T,\ell}=\emptyset$ if $\ell \neq n$.
Note that $ \gamma(n-k,x) x^k =x^n$ if $k=n$.
Thus $\sum_{\ell=k}^n |\pmset_{T,\ell}| x^\ell = \gamma(n-k,x) x^k$.

\bigskip

\noindent
{\bf Case 2} ($k=n-1$).
Let $T=[n] \setminus \{i\}$ where $i \in [n]$.
It then follows that
$\pmset_{T,n-1}=\{ (1,\dots,1) - {\bf e}_i - {\bf e}_{n+i}\}$
and $\pmset_{T,\ell}=\emptyset$ if $\ell \neq n-1$.
Note that $\gamma(n-k,x) x^k =x^{n-1} $ if $k=n-1$.
Thus $\sum_{\ell=k}^n |\pmset_{T,\ell}| x^\ell = \gamma(n-k,x) x^k$.

\bigskip

\noindent
{\bf Case 3} ($k=n-2$).
Let $T=[n] \setminus \{i,j\}$ where $1 \le i<j\le n$.
Since (\ref{wa2}) holds, 
each element of $\pmset_{T,\ell}$ is
$$\alpha_1 = (1,\dots,1) - {\bf e}_i - {\bf e}_j - {\bf e}_{n+i}- {\bf e}_{n+j}$$
if $\ell=n-2$, and  
is one of
\begin{eqnarray*}
\alpha_2 &=& (1,\dots,1) - {\bf e}_i - {\bf e}_{n+j},\\
\alpha_3 &=& (1,\dots,1)  - {\bf e}_j - {\bf e}_{n+i},
\end{eqnarray*}
if $\ell= n-1$.
Then 
each $\alpha_i$ corresponds to a perfectly matchable set.
In fact, a matching of $D(C_n)$ which corresponds to $\alpha_1$, $\alpha_2$, $\alpha_3$ is
$$\{ \{s, \overline{s}\} : s \in [n] \setminus \{i,j\}  \},$$
$$\{ \{s, \overline{s}\} : 1 \le s \le i-1 \mbox{ or } j+1 \le s \le  n  \} 
\cup \{ \{ s+1 , \overline{s}\} : i \le s \le j-1  \},$$
$$\{ \{s, \overline{s}\} : 1 \le s \le i-1 \mbox{ or } j+1 \le s \le  n  \} 
\cup \{ \{ s , \overline{s+1}\} : i \le s \le j-1  \},$$
respectively.
Thus
$$
\pmset_{T,\ell}=
\left\{
\begin{array}{cc}
\{\alpha_1\} & \mbox{if } \ell=n-2,\\
\\
\{\alpha_2, \alpha_3\}  & \mbox{if } \ell=n-1,\\
\\
\emptyset & \mbox{otherwise}.
\end{array}
\right.
$$
Note that $\gamma(n-k,x) x^k = (2x+1)  x^{n-2} =x^{n-2} + 2 x^{n-1}$ if $k=n-2$.
Thus $\sum_{\ell=k}^n |\pmset_{T,\ell}| x^\ell = \gamma(n-k,x) x^k$.

\bigskip

\noindent
{\bf Case 4} ($k\le n-3$).
Suppose that $T'=\{p,p+1,\dots,q\} \subset T$ and $p-1 , q+1 \notin T$.
Since $|T|\le n-3$, we have $n-q+p-1 \ge3$.
Let $U=\{ p,p+1,\dots,q, \overline{p}, \overline{p+1},\dots,\overline{q}\}$.
We will show that
there exists a perfectly matchable set $S$ of $D(C_{n-q+p-1})$ such that
$$\rho(S) =(x_1, \ldots, x_{p-1},x_{q+1},\dots,x_n,  y_1, \ldots, y_{p-1},y_{q+1},\dots,y_n)$$
for any $(x_1,\dots,x_n,y_1,\dots,y_n) \in \pmset_{T,\ell}$.
Here $C_{n-q+p-1}=(1,\ldots,p-1,q+1,\ldots,n,1)$
and the vertex set of $D(C_{n-q+p-1})$ is $\{1,\ldots,p-1,q+1,\ldots,n\} \cup \{\overline{1},\ldots,\overline{p-1},\overline{q+1},\ldots,\overline{n}\}$.
Let $M$ be a matching of $D(C_n)$ which corresponds to $(x_1,\dots,x_n,y_1,\dots,y_n)$.

\bigskip

\noindent
{\bf Case 4.1} (either $x_{p-1}=x_{q+1}=0$ or $y_{p-1}=y_{q+1}=0$).
Exchanging $[n]$ and $[\overline{n}]$ if needed, we may assume that 
$y_{p-1}=y_{q+1}=0$.
Then, for the subset $T' \subset [n]$,
$$
\sum_{i \in T'} x_i = \sum_{\overline{j} \in \Gamma_{G'}(T')} y_j = |T'|.
$$
Hence
the matching $M$ is the union of
a perfect matching of the induced subgraph $D(C_n)_U$ of $D(C_n)$
and a matching $M'$ of 
the induced subgraph $D(C_n)_{V\setminus U}$ of $D(C_n)$.
Then one can regard $M'$ as a matching of $D(C_{n-q+p-1})$ since $D(C_n)_{V\setminus U}$ is a subgraph of $D(C_{n-q+p-1})$.

\bigskip

\noindent
{\bf Case 4.2} ($x_{p-1}=y_{q+1} \neq y_{p-1}=x_{q+1}$).
Exchanging $[n]$ and $[\overline{n}]$ if needed, we may assume that 
$x_{p-1}=y_{q+1} =1$ and $y_{p-1}=x_{q+1}=0$.
If $e = \{q+2, \overline{q+1} \}$ belongs to the matching $M$,
then $M \setminus \{e\}$ is the union of
a perfect matching of $D(C_n)_U$ and a matching $M'$ of $D(C_n)_{V\setminus U}$
by the same argument in Case~4.1.
Suppose that $e = \{q, \overline{q+1} \}$ belongs to the matching $M$.
It then follows that $M$ is the union of
$\{ \{p-1, \overline{p} \} , \ldots,   \{q-1, \overline{q} \}  , \{q, \overline{q+1}\}\}$ and 
a matching $M'$ of $D(C_n)_{V\setminus U}$.
Thus one can regard $M' \cup \{\{p-1, \overline{q+1} \} \}$ as a matching of $D(C_{n-q+p-1})$.

\bigskip

Thus there exists a perfectly matchable set $S_1$ of $D(C_{n-q+p-1})$ such that
$$\rho(S_1) =(x_1, \ldots, x_{p-1},x_{q+1},\dots,x_n,  y_1, \ldots, y_{p-1},y_{q+1},\dots,y_n)$$
for  $(x_1,\dots,x_n,y_1,\dots,y_n)$ if $T'=\{p,p+1,\dots,q\} \subset T$ and $p-1 , q+1 \notin T$.
If, in addition, $T''=\{p',p'+1,\dots,q'\} \subset T$ and $p'-1 , q'+1 \notin T$ for some $p' > q+1$, then
there exists a perfectly matchable set $S_2$ of $D(C_{n-(q-p+1) -(q'-p'+1) })$ such that
$$\rho(S_2) =(x_1, \ldots, x_{p-1},x_{q+1},\dots, x_{p'-1},x_{q'+1},\dots, x_n,  y_1, \ldots, y_{p-1},y_{q+1},\dots, y_{p'-1},y_{q'+1},\dots,y_n)$$
by the same argument as above.
Repeating the above argument,  
it follows that there exists a perfectly matchable set $S$ of $D(C_{n-k})$ such that
$\rho(S) =(x_{i_1}, \ldots, x_{i_{n-k}},  y_{i_1}, \ldots, y_{i_{n-k}})$
where $[n] \setminus T = \{i_1,\ldots,i_{n-k}\}$ and $x_{i_r} + y_{i_r} \le 1$ for all $r$.
Then there exists a perfectly matchable set $S'$ of $C_{n-k}$
such that $|S'| = 2(\ell-k)$ and $\rho(S') = (x_{i_1} + y_{i_1}, \ldots, x_{i_{n-k}} + y_{i_{n-k}})$.
The matching corresponding to $S'$ is not unique if and only if 
$n-k$ is even and $\rho(S')=(1,\dots,1)$.
There exist two matchings for
the perfectly matchable set $S$ of $D(C_{n-k})$ 
exactly when 
$$(x_{i_1}, \ldots, x_{i_{n-k}},  y_{i_1}, \ldots, y_{i_{n-k}})
=(1,0,\ldots,1,0, 0,1,\ldots, 0,1),
(0,1,\ldots,0,1,1,0, \ldots, 1,0).
$$
Conversely, for each matching $M$ of $C_{n-k}$,
there exist $2^{|M|}$ vectors $(x_{i_1}, \ldots, x_{i_{n-k}},  y_{i_1}, \ldots, y_{i_{n-k}})$
where $x_{i_r} + y_{i_r} \le 1$ for all $r$ and
associated with a perfectly matchable set of $D(C_{n-k})$
since there are two possibilities $(x_{i_p}, x_{i_{p+1}},y_{i_p}, y_{i_{p+1}})
=(1,0,0,1), (0,1,1,0)$ for each $\{i_p, i_{p+1}\} \in M$.

Hence we have
$$
|\pmset_{T,\ell}| =
\left\{
\begin{array}{ll}
2^{\ell-k} m_{C_{n-k}} (\ell-k) -2 & \mbox{if } 
\ell -k = (n-k)/2,\\ 
\\
2^{\ell-k} m_{C_{n-k}} (\ell-k)& \mbox{otherwise}.
\end{array}
\right.
$$
Recall that $m_{G} (k)$ is the number of $k$-matchings of a graph $G$.
Thus
\begin{eqnarray*}
\sum_{\ell=k}^n |\pmset_{T,\ell}| x^\ell
&=&
\sum_{\ell=k}^n
2^{\ell-k} m_{C_{n-k}} (\ell-k) x^\ell -2 x^{(n+k)/2} \\
&=& \left( \sum_{\ell=k}^n m_{C_{n-k}} (\ell-k)(2 x)^{\ell-k} -2 x^{(n-k)/2} \right) x^k\\
&=&\gamma(n-k,x) x^k
\end{eqnarray*}
if $n-k$ is even, and
\begin{eqnarray*}
\sum_{\ell=k}^n |\pmset_{T,\ell}| x^\ell
&=&
\sum_{\ell=k}^n
2^{\ell-k} m_{C_{n-k}} (\ell-k) x^\ell\\
&=&\left( \sum_{\ell=k}^nm_{C_{n-k}} (\ell-k)(2 x)^{\ell-k}\right) x^k\\
&=&\gamma(n-k,x) x^k
\end{eqnarray*}
if $n-k$ is odd.

Therefore the $h^*$-polynomial of $\nabla_{W_n}^{\rm PQ}$ is 
$$
\sum_{T \subset [n]}
\sum_{\ell=|T|}^n |\pmset_{T,\ell}| x^\ell
=
\sum_{T \subset [n]}
\gamma(n-|T|) x^{|T|}
=
\sum_{k=0}^n \binom{n}{k} \gamma(n-k,x) x^k.
$$
Moreover
\begin{eqnarray*}
 & & \sum _{k=0}^n \binom{n}{k} \gamma(n-k,x) x^k\\
&=&
x^n+\sum_{k=0}^n \binom{n}{k} 
\left(\frac{1+\sqrt{1+8x}}{2}\right)^{n-k} x^k+
\sum_{k=0}^n \binom{n}{k} \left(\frac{1-\sqrt{1+8x}}{2}\right)^{n-k}
x^k\\
& &
- \sum_{\substack{k\ge0\\ k \equiv n \ (\mbox{\small mod } 2)}} 2  \binom{n}{k} \sqrt{x}^{n-k} x^k
\\
&=&
\left(\frac{1+2x+\sqrt{1+8x}}{2}\right)^n +
\left(\frac{1+2x-\sqrt{1+8x}}{2}\right)^n 
+x^n
-
(x+\sqrt{x})^n 
- (x-\sqrt{x})^n.
\end{eqnarray*}
In particular, substituting $x=1$, we have
$$
\sum _{k=0}^n \binom{n}{k} \gamma(n-k,1)
=3^n-2^n+1.
$$

\end{proof}

\section*{Acknowledgement}
The authors were partially supported by JSPS KAKENHI 18H01134, 19K14505, and 19J00312.

\end{document}